\documentclass[10pt,a4paper]{article}

\usepackage{amsmath,amssymb,amsthm,amsfonts}
\usepackage{geometry}
\usepackage{hyperref}
\hypersetup{colorlinks=true,linkcolor=blue,citecolor=blue,urlcolor=blue}
\usepackage{mathrsfs}
\usepackage{comment}

\usepackage{ulem}
\numberwithin{equation}{section}

\usepackage[style=numeric-comp,sorting=nyt,backend=biber,maxbibnames=99]{biblatex}
\allowdisplaybreaks[3]

\newtheorem{theorem}{Theorem}[section]
\newtheorem{remark}{Remark}[section]

\newtheorem{lemma}[theorem]{Lemma}
\newtheorem{corollary}[theorem]{Corollary}

\newcommand{\curl}{\mathop{\mathrm{curl}}\nolimits}
\newcommand{\divv}{\mathop{\mathrm{div}}\nolimits}
\newcommand{\dx}{\,dx}
\newcommand{\dt}{\,dt}
\newcommand{\norm}[1]{\left\| #1 \right\|}

\newcommand{\pd}[1]{\partial_{#1}}
\newcommand{\od}[1]{\frac{d}{d #1}}
\newcommand{\Linf}{{L}^{\infty}}
\newcommand{\Ltwo}{{L}^{2}}
\newcommand{\Hone}{H^{1}}
\newcommand{\Htwo}{H^{2}}

\newcommand{\Cone}{C_{1}}
\newcommand{\Ctwo}{C_{2}}
\newcommand{\Cthree}{C_{3, T, M}}
\newcommand{\Cfour}{C_{4, T, M}}
\newcommand{\Cfive}{C_{5, T, M}}

\begin{document}

\title{Conditonal Lipschitz stability for the Inverse Problem of the 2D Navier-Stokes System in a Bounded Domain\\
\medskip
\small{\it Dedicated to the 60th anniversary of Professor Roman Novikov
}}
\author{Jishan Fan\thanks{Department of Applied Mathematics, Nanjing Forestry University, Nanjing 210037, China. Email: \mbox{fanjishan@njfu.edu.cn}} \and
Yu Jiang\thanks{School of Mathematics, Shanghai University of Finance and Economics, Shanghai 200433, China. Email: \mbox{jiang.yu@mail.shufe.edu.cn}} \and
Sei Nagayasu\thanks{Graduate School of Sciences, University of Hyogo, Himeji 671-2201, Japan. Email: \mbox{sei@sci.u-hyogo.ac.jp}} \and
Gen Nakamura\thanks{Department of Mathematics, Hokkaido University, Sapporo 060-0808, Japan, Research Center of Mathematics for Social Creativity, Research Institute for Electronic Science, Hokkaido University, Sapporo 060-0812, Japan. Email: \mbox{nakamuragenn@gmail.com}}}
\date{}
\maketitle

\begin{abstract}
This paper concerns an inverse problem for the initial boundary value problem of the two-dimensional Navier-Stokes system defined in a bounded simply connected domain with slip, vorticity boundary conditions, and a global vorticity invariant constraint. We establish conditional Lipschitz stability and a local recovery for this inverse problem, where the velocity field and space-independent boundary vorticity are locally recovered from the given initial velocity field and the global vorticity invariant. Our analysis is based on well-posedness estimates and energy methods for the vorticity transport equation.
\end{abstract}

\noindent \textbf{Running Title}: Conditional stability for the 2D Navier-Stokes Inverse Problem

\section{Introduction}
This paper concerns an inverse problem for the initial boundary value problem of the two-dimensional Navier-Stokes system defined in a bounded simply connected domain with slip, vorticity
boundary conditions, and a global vorticity invariant constraint, which is given as
\begin{equation}\label{eq:1.1}
\pd{t} u + u\cdot \nabla u + \nabla p - \Delta u = 0, \quad \divv u = 0 \quad \text{in } \Omega \times (0,T),
\end{equation}
\begin{equation}\label{eq:1.2}
u \cdot n = 0, \quad \omega := \curl u = h(t) \quad \text{on } \partial \Omega \times (0,T),
\end{equation}
\begin{equation}\label{eq:1.4}
u(\cdot, 0) = u_0 \quad \text{in } \Omega,
\end{equation}
\begin{equation}\label{eq:1.3}
\frac{1}{|\Omega|} \int_{\Omega} \omega \dx = L \quad \text{in } (0, T).
\end{equation}
Here, $\Omega \subset \mathbb{R}^2$ with area $|\Omega|$ and smooth boundary $\partial \Omega$ denotes a bounded simply connected domain. We denote the outward unit normal vector of $\partial\Omega$ by $n$. Also, the fluid velocity field $u$ is a column vector $u=(u_1,u_2)^{\mathfrak{t}}$, and we abbreviate $(u\cdot\nabla)u$ as $u\cdot\nabla u$. Further, $p$ and $\omega := \curl u = \partial_1 u_2 - \partial_2 u_1$ describe the pressure field and the scalar vorticity, respectively. The boundary condition in \eqref{eq:1.2} $u \cdot n = 0$ is {the well-known Navier slip boundary condition, a standard rigid-wall impermeability constraint that prevents fluid from penetrating across the boundary.} In contrast, condition $\omega = h(t)$ prescribes boundary-generated vorticity, a quantity that is often unmeasurable or unknown in practical hydrodynamic experiments (see \cite{ref10}). $L$ denotes the global vorticity invariant. \eqref{eq:1.3} determines the spatial average of the vorticity field, a physical invariant associated with total circulation $|\Omega| L$, which is conserved in inviscid 2D flows according to Kelvin's Circulation Theorem (see \cite{saffman1992vortex,chorin1994vorticity,ref8}). This quantity is typically accessible through bulk measurements in experimental flow setups (see \cite{raffel2018particle}).

From the perspective of inverse problems, the mentioned initial boundary value problem is called a forward problem, and its objective is to determine the velocity field $u$ and pressure field $p$ when the boundary condition \eqref{eq:1.2} and initial condition \eqref{eq:1.4} are given. More precisely, the following is known for the unique solvability of the forward problem (see \cite{ref8}).

\begin{theorem}[Forward problem]\label{th:fp}
Let $u_0(x)\in H^1(\Omega)$, $h(t)\in L^\infty(0,T)$, and $h'(t) \in L^2(0,T)$ with $u_0\cdot n=0$, $\curl u_0=h(0)$ on $\partial\Omega$. Then, there exists a unique strong solution $(u,p)$ to \eqref{eq:1.1}--\eqref{eq:1.4} such that
$u \in \Linf(0, T ; \Hone(\Omega)) \cap \Ltwo(0, T ; \Htwo(\Omega))$, $\partial_t u \in \Ltwo(0, T ; \Ltwo(\Omega))$, and $p\in L^2(0,T;H^1(\Omega))$ with $\int_\Omega p \,dx=0$.
\end{theorem}

\medskip
Now, we shift our perspective to the {inverse problem} associated with the 2D Navier-Stokes system \eqref{eq:1.1}--\eqref{eq:1.3}, addressing a distinct and physically meaningful question formulated as follows:

\medskip
\noindent \textbf{Inverse problem}: Given the initial velocity field $u_0$ and the global vorticity invariant $L$, identify the space-independent boundary vorticity $h(t)$, together with the velocity field $u$ and pressure field $p$, such that the coupled system \eqref{eq:1.1}--\eqref{eq:1.3} is satisfied. 
\medskip

Inverse problems for the Navier-Stokes system have been the subject of extensive research over the past few decades. Most studies have focused on recovering unknown quantities, such as initial conditions, boundary forcing terms, or viscosity coefficients from partial observations of the flow field (see \cite{ref11,ref13,ref14}). For the 2D Navier-Stokes system, existing studies on inverse problems mostly restrict themselves to the case where the boundary vorticity is given (see \cite{ ref8}). The inverse problem of recovering the unknown space-independent boundary vorticity $h(t)$ was pointed out in \cite{ref3}. Our research was proposed by one of the authors of the paper and also significantly motivated the coauthors of this paper to work together. We point out that there is a similar inverse problem for the Ginzburg-Landau model, where unknown boundary magnetic terms are identified from the global flux invariant (see \cite{ref1,ref2}).

\bigskip
The primary goal of this paper is to prove the conditional Lipschitz stability estimate for the aforementioned inverse problem of the 2D Navier-Stokes system. This result directly implies the local recovery of the unknown boundary vorticity $h(t)$ from the measurable initial velocity field $u_0$, and the global vorticity is invariant. Our main result is stated as follows:

\begin{theorem}[Conditional Lipschitz stability estimate]\label{thm:1.1} Let $u_{0i} \in \Hone(\Omega)$ satisfy 
$\divv u_{0i} = 0$ in $\Omega$ 
and $u_{0i} \cdot n = 0$ on $\partial \Omega$, and let $( u_i, p_i)$ with $\int_\Omega p_i \, dx = 0$ be the strong solutions to the system \eqref{eq:1.1}--\eqref{eq:1.3} corresponding to the input data $(u_{0i}, L_i)$ 
for $i=1,2$. Assume that $\norm{\omega_{0i}}_{\Ltwo ( \Omega )} \leq M$ ($i=1,2$) holds for a constant $M>0$, 
where $\omega_{0i} = \curl u_{0i}$. 
Then we have the Lipschitz stability estimate:
\begin{equation}\label{eq:1.5t}
\|u_1-u_2\|_{L^\infty(0,T;H^1(\Omega))}
+ \|p_1-p_2\|_{L^2(0,T;H^1(\Omega))}
+ \|h_1-h_2\|_{L^2(0,T)}
\le C \, \norm{\omega_{01} - \omega_{02}}_{\Ltwo(\Omega)}
\end{equation}
for any $T>0$, where $C$ depends only on $\Omega, T$ and $M$.
\end{theorem}

\begin{corollary}[Local recovery]
By reformulating the inverse problem as a nonlinear operator equation $F(\xi)=\eta$ with unknowns $\xi=(u,p,h)$ and data $\eta=(u_0,L)$, conditional Lipschitz stability \eqref{eq:1.5t} ensures that the unknowns can be locally recovered via iterative regularization methods such as the Landweber method (see \cite{Kaltenbacher,Hstability}) or the Levenberg-Marquardt Method (see \cite{Ishida}).
\end{corollary}

The remainder of this paper is structured as follows. In Section 2, we first provide preliminary estimates that are essential to the proof of Theorem \ref{thm:1.1}. Section 3 is then devoted to the rigorous proof of Theorem \ref{thm:1.1}.

\section{Preliminary Estimates}

To prove Theorem \ref{thm:1.1}, we first establish a set of preliminary estimates, a standard and essential step in the analysis of nonlinear evolution equations. These estimates show that the recovered boundary vorticity $h(t)$ belongs to the space $\Ltwo(0, T)$, the velocity field $u$ lies in the space $\Linf(0, T ; \Hone(\Omega)) \cap \Ltwo(0, T ; \Htwo(\Omega))$, and the pressure $p$ lies in the space $L^2(0,T;H^1(\Omega))$. These function spaces are natural regularity classes for physically meaningful 2D fluid flows (see \cite{ref7,ref8}). Hereafter, we write $A \lesssim B$ for simplicity of notation to mean that $A \le C B$ with a constant $C>0$ depending only on the domain $\Omega$ 
\begin{lemma}\label{lemma:1}
Let $u_{0} \in \Hone(\Omega)$ satisfy 
$\divv u_{0} = 0$ in $\Omega$ 
and $u_{0} \cdot n = 0$ on $\partial \Omega$,
and let 
$( u, p)$ be the strong solutions to the system 
\eqref{eq:1.1}--\eqref{eq:1.3} corresponding to the input data $(u_{0}, L)$. 
Then for any $T > 0$, we have the following estimates:
\begin{align}
& \norm{\omega}_{\Linf (0,T; \Ltwo (\Omega))}
\leq \norm{\omega_{0}}_{\Ltwo ( \Omega )} , \quad 
\norm{\nabla \omega}_{\Ltwo (0,T; \Ltwo ( \Omega ))}
\leq \frac{1}{\sqrt{2}} \norm{\omega_{0}}_{\Ltwo ( \Omega )} , 
\label{eq:lemma11}\\
&\sup_{0 \leq t \leq T} \norm{u (t)}_{\Hone(\Omega)} 
\lesssim \norm{\omega_{0}}_{\Ltwo(\Omega)} ,\quad
\norm{u}_{\Ltwo(0,T;\Htwo(\Omega))}
\lesssim \max \{ \sqrt{T} , 1 \} \norm{\omega_{0}}_{\Ltwo(\Omega)},
\label{eq:lemma12} \\
& \norm{\pd{t} u}_{\Ltwo (0,T; \Ltwo ( \Omega ) )}
+ \norm{\nabla p}_{\Ltwo (0,T; \Ltwo ( \Omega ) )}\lesssim \max \{ \sqrt{T} , 1 \} \bigl( 
 \norm{\omega_{0}}_{\Ltwo ( \Omega )}^{2}
 + \norm{\omega_{0}}_{\Ltwo ( \Omega )}
\bigr) , 
\label{eq:lemma13} \\
& \norm{h}_{\Ltwo (0,T)}
\lesssim \max \{ \sqrt{T} , 1 \}\norm{\omega_{0}}_{\Ltwo(\Omega)}. 
\label{eq:lemma14}
\end{align}
\end{lemma}
\begin{proof}
Applying the curl operator to the momentum equation \eqref{eq:1.1}, we have
\begin{equation*}
\pd{t} \omega + u \cdot \nabla \omega 
+ \omega \divv u - \Delta \omega = 0 . 
\end{equation*}
This equation and the divergence-free condition on $u$ yield the vorticity transport equation:
\begin{equation}\label{eq:1.6}
\pd{t} \omega + u \cdot \nabla \omega - \Delta \omega = 0.
\end{equation}
Since the boundary vorticity $h(t)$ is independent of the spatial variable $x$, we can rewrite \eqref{eq:1.6} as
\begin{equation}\label{eq:1.7}
\pd{t} \omega + u \cdot \nabla(\omega - h) - \Delta(\omega - h) = 0.
\end{equation}
Testing the equation \eqref{eq:1.7} with $\omega - h$, 
we have  
{
\begin{equation}\label{eq:002}
\int_{\Omega} ( \omega - h ) ( \pd{t} \omega ) \dx 
+ \int_{\Omega} ( \omega - h ) u \cdot \nabla ( \omega - h ) \dx 
- \int_{\Omega} ( \omega - h ) \Delta ( \omega -h ) \dx = 0 .
\end{equation}
As for the first term of \eqref{eq:002}, we have 
\[
\int_{\Omega} ( \pd{t} \omega )  h \dx
= \left( \frac{d}{d t} \int_{\Omega} \omega \dx \right) h 
= \frac{d}{d t} \bigl( L | \Omega | \bigr) h = 0 \cdot h = 0  
\]
by using \eqref{eq:1.3}. In the second term of \eqref{eq:002}, 
by integration by parts, and using \eqref{eq:1.2} and the divergence-free condition on $u$, we have 
\begin{align}
\int_{\Omega} ( \omega - h ) u \cdot \nabla ( \omega - h ) \dx 
& = \frac{1}{2} \int_{\Omega} 
  u\cdot \nabla ( \omega - h )^2\dx \notag \\
& = \frac{1}{2} \left(
\int_{\partial \Omega} u \cdot n ( \omega - h )^{2} \, d \sigma
- \int_{\Omega} ( \divv u ) ( \omega - h )^{2} \dx
\right)= 0 . \label{eq:009}
\end{align}
As for the third term, 
by integration by parts, and using the second condition in \eqref{eq:1.2}
and the independence of $h (t)$ with respect to $x$, we have 
\[
\int_{\Omega} ( \omega - h ) \Delta ( \omega -h ) \dx 
= \int_{\partial \Omega}
 ( \omega - h ) n \cdot \nabla ( \omega - h ) \, 
d \sigma 
- \int_{\Omega}|\nabla(\omega - h)|^2 \dx 
= - \int_{\Omega}|\nabla \omega|^2 \dx . 
\]
Summing up,} 
we have derived the key energy identity for the vorticity field:
\begin{equation}\label{eq:1.9}
\frac{1}{2} \od{t} \int_{\Omega} \omega^2 \dx 
+ \int_{\Omega}|\nabla \omega|^2 \dx = 0.
\end{equation}
{Integrating \eqref{eq:1.9} with respect to $t$, we have 
\begin{equation}\label{eq:003}
\frac{1}{2} \int_{\Omega} \omega^{2}(t) \dx 
+ \int_{0}^{t} \int_{\Omega} | \nabla \omega |^{2} \dx \dt 
= \frac{1}{2} \int_{\Omega} \omega_{0}^{2} \dx , 
\end{equation}
which yields the estimates \eqref{eq:lemma11}. 
}%
Since $\Omega$ is a bounded simply connected domain and $u$ satisfies the divergence-free condition and the first condition of \eqref{eq:1.2}, by virtue of the classical inequalities in \cite{ref4,ref5}, we obtain the following regularity estimates relating the velocity field and the vorticity field:
{
\begin{align}
& \norm{u (t)}_{\Hone(\Omega)} 
\lesssim \norm{\omega (t)}_{\Ltwo(\Omega)},
\label{eq:1.10t} \\
& \norm{u (t)}_{\Htwo(\Omega)} 
\lesssim \norm{\omega (t)}_{\Hone(\Omega)} .
\label{eq:1.11t}
\end{align}
By \eqref{eq:1.10t} and \eqref{eq:003}, we have the estimate 
\begin{equation}\label{eq:004}
\sup_{0 \leq t \leq T} \norm{u (t)}_{\Hone(\Omega)} 
\lesssim \sup_{0 \leq t \leq T} \norm{\omega (t)}_{\Ltwo(\Omega)} 
\leq \norm{\omega_{0}}_{\Ltwo(\Omega)} . 
\end{equation}
In the same way, by \eqref{eq:1.11t} and \eqref{eq:003}, 
we have 
\begin{align}
\norm{u}_{\Ltwo(0,T;\Htwo(\Omega))}^{2}
& = \int_{0}^{T} \norm{u (t)}_{\Htwo(\Omega)}^{2} \dt 
\lesssim \int_{0}^{T} \norm{\omega (t)}_{\Hone(\Omega)}^{2} \dt 
\notag \\
& = \int_{0}^{T} 
 \norm{\omega (t)}_{\Ltwo(\Omega)}^{2} 
\dt 
+ \int_{0}^{T} 
 \norm{\nabla \omega (t)}_{\Ltwo(\Omega)}^{2} 
\dt \notag \\
& \leq \int_{0}^{T} 
 \norm{\omega_{0}}_{\Ltwo(\Omega)}^{2} 
\dt 
+ \frac{1}{2} \norm{\omega_{0}}_{\Ltwo(\Omega)}^{2} 
= \left( T + \frac{1}{2} \right) 
\norm{\omega_{0}}_{\Ltwo(\Omega)}^{2} . 
\label{eq:005}
\end{align}
By the estimates \eqref{eq:004} and \eqref{eq:005}, 
we obtain \eqref{eq:lemma12}. 
}%

From \eqref{eq:1.1}, we obtain 
$$\partial_t u +\nabla p=\Delta u-u\cdot \nabla u,$$
this implies
$$\int_0^T \int_\Omega |\partial_t u +\nabla p|^2 \dx\dt = \int_0^T \int_\Omega |\Delta u-u\cdot \nabla u|^2 \dx\dt.$$
On the other hand, using $\mathrm{div} u=0$ in $\Omega$ and $u\cdot n=0$ on $\partial \Omega$, we derive the identity
\begin{equation}\label{eq:1.13}
\int_{\Omega} \nabla p \cdot \pd{t} u \, \dx= \int_{\Omega} 
 \bigl( \divv ( p \pd{t} u ) - p \pd{t} ( \divv u ) \bigr) 
\dx= \int_{\Omega} \divv\left(p \pd{t} u\right) \dx = 0.
\end{equation}
As a consequence, we have 
the time derivative of the velocity field
and the pressure gradient
\begin{align*}
\int_0^T \int_\Omega \bigl( |\partial_t u|^2 +|\nabla p|^2 \bigr) \dx\dt
& = \int_0^T \int_\Omega |\partial_t u +\nabla p|^2 \dx \dt
= \int_0^T \int_\Omega |\Delta u-u\cdot \nabla u|^2 \dx \dt \\
& \leq 2 \left( 
 \int_{0}^{T} \int_{\Omega}|u \cdot \nabla u|^2\dx\dt + \int_{0}^{T} \int_{\Omega}|\Delta u|^2\dx\dt 
\right) .
\end{align*}
We now estimate the two terms on the right-hand side of these estimates.
By the Sobolev embedding theorem, we have 
\[
\norm{u}_{L^{\infty} ( \Omega )} 
\lesssim \norm{u}_{\Htwo ( \Omega )} . 
\]
Hence, by \eqref{eq:004} and \eqref{eq:005}, we obtain 
\begin{align*}
\int_{0}^{T} \int_{\Omega} | u \cdot \nabla u |^{2} \dx \dt 
& \leq \int_{0}^{T} \norm{u}_{L^{\infty} ( \Omega )}^{2} 
\int_{\Omega} | \nabla u |^{2} \dx \dt 
\lesssim \int_{0}^{T} \norm{u}_{\Htwo ( \Omega )}^{2} 
\int_{\Omega} | \nabla u |^{2} \dx \dt \notag \\
& \lesssim \int_{0}^{T} \norm{u}_{\Htwo ( \Omega )}^{2} \dt 
\norm{\omega_{0}}_{\Ltwo ( \Omega )}^{2} 
\lesssim \max \{ T, 1 \} 
\norm{\omega_{0}}_{\Ltwo ( \Omega )}^{4} . 
\end{align*}
On the other hand, by \eqref{eq:005} we have 
\begin{equation*}
\int_{0}^{T} \int_{\Omega} | \Delta u |^{2} \dx \dt 
\lesssim \max \{ T, 1 \} 
\norm{\omega_{0}}_{\Ltwo ( \Omega )}^{2} . 
\end{equation*}
We therefore conclude that the estimate \eqref{eq:lemma13} holds.

Finally, we establish the $\Ltwo(0,T)$ regularity for the boundary vorticity $h(t)$ through a series of embedding and energy estimates:
\[
\int_{0}^{T} h^{2} (t) \dt 
= \frac{1}{| \Omega |} \int_{0}^{T} \int_{\Omega} h^{2} (t) \dx \dt 
\leq \frac{2}{| \Omega |} 
\left(
 \int_{0}^{T} \int_{\Omega} \omega^{2} \dx \dt 
 + \int_{0}^{T} \int_{\Omega} ( \omega - h )^{2} \dx \dt 
\right) . 
\]
On the second term of the right-hand side, 
by the second condition of \eqref{eq:1.2} 
and the Poincar{\'e} inequality, we have 
\[
\int_{\Omega} ( \omega - h )^{2} \dx 
\lesssim \int_{\Omega} | \nabla ( \omega - h ) |^{2} \dx .
\]
Together with \eqref{eq:003}, we have 
\begin{align*}
\int_{0}^{T} h^{2} (t) \dt 
& \lesssim \int_{0}^{T} \int_{\Omega} \omega^{2} \dx \dt 
+ \int_{0}^{T} \int_{\Omega} | \nabla ( \omega - h ) |^{2} \dx \dt \\
& = \int_{0}^{T} \int_{\Omega} \omega^{2} \dx \dt 
+ \int_{0}^{T} \int_{\Omega} | \nabla \omega |^{2} \dx \dt 
\leq \left( T + \frac{1}{2} \right)
\norm{\omega_{0}}_{\Ltwo ( \Omega )}^{2} .
\end{align*}
Hence, we have obtained the estimate \eqref{eq:lemma14}.
\end{proof}
\begin{lemma}\label{lemma:2}
Under the same assumption of Lemma~\ref{lemma:1}, 
the estimate
\begin{equation}\label{eq:lemma2}
\int_{0}^{T} \norm{\omega (t)}_{L^{4} ( \Omega )}^{2} \dt 
\lesssim \max\{ T, 1 \} \norm{\omega_{0}}_{\Ltwo ( \Omega )}^{2} 
\end{equation}
holds. 
\end{lemma}
\begin{proof}
By the Sobolev embedding theorem, 
we have 
\[
\norm{\omega (t)}_{L^{4} ( \Omega )}^{2}
\lesssim \norm{\omega (t)}_{\Hone ( \Omega )}^{2} . 
\]
Then by using \eqref{eq:lemma11} we have 
\begin{align*}
\int_{0}^{T} \norm{\omega (t)}_{L^{4} ( \Omega )}^{2} \dt
\lesssim \int_{0}^{T} \norm{\omega (t)}_{\Hone ( \Omega )}^{2} \dt
& = \int_{0}^{T} \norm{\omega (t)}_{\Ltwo ( \Omega )}^{2} \dt
+ \int_{0}^{T} \norm{\nabla \omega (t)}_{\Ltwo ( \Omega )}^{2} \dt \\
& \leq 
\left( T + \frac{1}{2} \right) 
\norm{\omega_{0}}_{\Ltwo ( \Omega )}^{2} , 
\end{align*}
which means \eqref{eq:lemma2}. 
\end{proof}
This completes the derivation of all necessary preliminary estimates for the proof of Theorem \ref{thm:1.1}. Before closing this section, we give the following remark on the large time behavior of the vorticity field.
\begin{remark}
By the Poincar\'e inequality, we have
\[
\norm{\omega - \frac{1}{|\Omega|} \int_{\Omega} \omega \dx}_{\Ltwo(\Omega)} 
\leq C \norm{\nabla \omega}_{\Ltwo(\Omega)} ,
\]
where a positive constant $C$ depends only on $\Omega$. Using this estimate, the energy identity \eqref{eq:1.9} and the identity
\[
\od{t} \int_{\Omega} \omega^2 \dx = \od{t} \int_{\Omega}\left(\omega - \frac{1}{|\Omega|} \int_{\Omega} \omega \dx\right)^2 \dx , 
\]
we have
\[
\od{t} \norm{\omega - \frac{1}{|\Omega|} \int_{\Omega} \omega \dx}_{\Ltwo(\Omega)}^{2}
\leq - \frac{2}{C^{2}}
\norm{\omega - \frac{1}{|\Omega|} \int_{\Omega} \omega \dx}_{\Ltwo(\Omega)}^{2} . 
\] 
It follows that the vorticity field decays exponentially in time to the centered vorticity field (the spatial average of the vorticity field):
\[
\norm{\omega - \frac{1}{|\Omega|} \int_{\Omega} \omega \dx}_{\Ltwo(\Omega)}
\leq \norm{\omega_{0} - \frac{1}{|\Omega|} \int_{\Omega} \omega_{0} \dx}_{\Ltwo(\Omega)}
e^{-\lambda_0 t} , 
\]
where $\lambda_0 = 1 / C^2$. 
\end{remark}

\section{Proof of Theorem \ref{thm:1.1}}
This section is devoted to proving Theorem \ref{thm:1.1}, the main result regarding the uniqueness and conditional stability presented in this paper.

Let $(u_i, p_i, h_i)$ be the strong solutions to the system \eqref{eq:1.1}--\eqref{eq:1.3} corresponding to the input data $(u_{0i}, L_i)$ for $i=1,2$. From the vorticity equation \eqref{eq:1.6}, we obtain 
\begin{equation}\label{eq:2.20}
\pd{t}(\omega_1 - \omega_2) + u_1\cdot \nabla (\omega_1 - \omega_2) + (u_1 - u_2)\cdot \nabla \omega_2 - \Delta (\omega_1 - \omega_2) = 0.
\end{equation}
Since $h_i(t)\ (i=1,2)$ is independent of the spatial variable $x$, we rewrite \eqref{eq:2.20} in terms of the difference $\left(\omega_1 - \omega_2\right) - \left(h_1 - h_2\right)$:
\begin{equation*}
\pd{t}(\omega_1 - \omega_2) + u_1 \cdot \nabla\left[\left(\omega_1 - \omega_2\right)-\left(h_1 - h_2\right)\right] + \left(u_1 - u_2\right) \cdot \nabla \omega_2 - \Delta\left[\left(\omega_1 - \omega_2\right)-\left(h_1 - h_2\right)\right] = 0.
\end{equation*}
We now test the above equation with $(\omega_1 - \omega_2)-(h_1 - h_2)$. 
As for the first term of the left-hand side, we obtain the key identity
\begin{align*}
\int_{\Omega} 
 \pd{t} ( \omega_1 - \omega_2 ) \cdot ( h_1 - h_2 ) 
\dx 
& = ( h_1 - h_2 ) 
\od{t} \int_{\Omega} ( \omega_1 - \omega_2 ) \dx \\
& = ( h_1 - h_2 ) \od{t}
\bigl( | \Omega | L_{1} - | \Omega | L_{2} \bigr) 
= 0,
\end{align*}
which holds due to the global vorticity invariant constraint \eqref{eq:1.3}. 
As for the second term, 
we proceed in the same way as we derived the identity \eqref{eq:009}.  
By integrating by parts, and using the first condition in \eqref{eq:1.2} and the divergence-free condition on $u_{i}$, we have 
\begin{align*}
& \int_{\Omega} 
 \bigl( u_1 \cdot \nabla 
 \left[ \left( \omega_1 - \omega_2 \right ) - \left( h_1 - h_2 \right) \right]
 \Bigr) 
 \left[ \left( \omega_1 - \omega_2 \right) - \left( h_1 - h_2 \right) \right] 
\dx \\
& = \frac{1}{2} 
\int_{\Omega} 
  u_1 \cdot \nabla 
 \left[ \left( \omega_1 - \omega_2 \right ) - \left( h_1 - h_2 \right) \right]^{2}
\dx \\
& = \frac{1}{2} \left(
 \int_{\partial \Omega} 
  u_1 \cdot n 
 \left[ \left( \omega_1 - \omega_2 \right ) - \left( h_1 - h_2 \right) \right]^{2}
 \, d \sigma
- \int_{\Omega} 
 ( \divv u_1 )  
 \left[ \left( \omega_1 - \omega_2 \right ) - \left( h_1 - h_2 \right) \right]^{2}
\dx 
\right) = 0 .  
\end{align*}
As for the fourth term, 
by integration by parts, and using the second condition in \eqref{eq:1.2}
and the independence of $h_{i} (t)$ with respect to $x$, we have 
\begin{align*}
& \int_{\Omega} 
 \bigl( 
  \Delta \left[\left(\omega_1 - \omega_2\right)-\left(h_1 - h_2\right)\right] 
 \bigr) 
 \left[\left(\omega_1 - \omega_2\right)-\left(h_1 - h_2\right)\right]
d x \\
& = - \int_{\Omega} 
 \left|
  \nabla \left[\left(\omega_1 - \omega_2\right)-\left(h_1 - h_2\right)\right]
 \right|^{2}
\dx = - \int_{\Omega} 
 \left|
  \nabla ( \omega_1 - \omega_2 )
 \right|^{2}
\dx .
\end{align*}
Summing up, we have obtained
\begin{align}
& \frac{1}{2} \od{t} \int_{\Omega}\left(\omega_1 - \omega_2\right)^2 \dx + \int_{\Omega}\left|\nabla\left(\omega_1 - \omega_2\right)\right|^2 \dx 
\notag \\
& = -\int_{\Omega}
 (u_1 - u_2) \cdot \nabla 
 \omega_2\left[(\omega_1 - \omega_2)-(h_1 - h_2)\right] 
\dx . \label{eq:015}
\end{align}
We now estimate the right-hand side. 
We first apply integration by parts to the right-hand side. 
By the second condition in \eqref{eq:1.2}, 
the divergence-free condition on $u_{i}$ 
and the independence of $h_{i} (t)$ with respect to $x$, we have
\begin{align}
& \mbox{$-$} \int_{\Omega} 
 ( ( u_{1} - u_{2} ) \cdot \nabla \omega_{2} )
 \left[\left(\omega_1 - \omega_2\right)-\left(h_1 - h_2\right)\right] 
\dx \notag \\
& = \int_{\Omega} 
 \omega_{2} 
 \divv \bigl( 
  ( u_{1} - u_{2} )  
  \left[\left(\omega_1 - \omega_2\right)-\left(h_1 - h_2\right)\right] 
 \bigr) 
\dx \notag \\
& = \int_{\Omega} 
 \omega_{2} 
  \divv ( u_{1} - u_{2} )  
  \left[\left(\omega_1 - \omega_2\right)-\left(h_1 - h_2\right)\right] 
\dx 
+ \int_{\Omega} 
  ( u_{1} - u_{2} ) \omega_{2} \cdot 
  \nabla \left[\left(\omega_1 - \omega_2\right)-\left(h_1 - h_2\right)\right] 
\dx \notag \\
& = \int_{\Omega} 
  ( u_{1} - u_{2} ) \omega_{2} \cdot 
  \nabla \left(\omega_1 - \omega_2\right)
\dx . \label{eq:016}
\end{align}
Applying H\"{o}lder's inequality, 
we have
\begin{equation}\label{eq:017}
\int_{\Omega} 
 ( u_1 - u_2 ) \omega_2 \cdot \nabla (\omega_1 - \omega_2 ) \dx 
\leq \norm{u_1 - u_2}_{L^4(\Omega)} 
\norm{\omega_2}_{L^4(\Omega)} 
\norm{\nabla (\omega_1 - \omega_2 )}_{\Ltwo(\Omega)} . 
\end{equation}
Since $u_1 - u_2$ is divergence-free and 
$( u_1 - u_2 ) \cdot n = 0$ on $\Omega$, we have
\begin{equation}\label{eq:018}
\norm{u_1 - u_2}_{L^4(\Omega)} 
\leq \mbox{{$\Cone$}} \norm{\omega_1 - \omega_2}_{\Ltwo(\Omega)} 
\end{equation}
by the Sobolev embedding theorem,
where {$\Cone$} is a positive constant 
depending only on $\Omega$. 
Then, by \eqref{eq:015}, \eqref{eq:016}, 
\eqref{eq:017} and \eqref{eq:018}, we have 
\begin{align*}
\frac{1}{2} \frac{d}{d t} 
\int_{\Omega} ( \omega_{1} - \omega_{2} )^{2} \dx 
+ \int_{\Omega} 
 \left|
  \nabla ( \omega_1 - \omega_2 )
 \right|^{2}
\dx 
& \leq \norm{u_1 - u_2}_{L^4(\Omega)} 
\norm{\omega_2}_{L^4(\Omega)} 
\norm{\nabla\left(\omega_1 - \omega_2\right)}_{\Ltwo(\Omega)} \\
& \leq \Cone \norm{\omega_1 - \omega_2}_{L^2(\Omega)} 
\norm{\omega_2}_{L^4(\Omega)} 
\norm{\nabla\left(\omega_1 - \omega_2\right)}_{\Ltwo(\Omega)} \\
& \leq \frac{1}{2} 
\norm{\nabla\left(\omega_1 - \omega_2\right)}_{\Ltwo(\Omega)}^{2}
+ \frac{\Cone^{2}}{2} 
\norm{\omega_1 - \omega_2}_{L^2(\Omega)}^{2}
\norm{\omega_2}_{L^4(\Omega)}^{2} , 
\end{align*}
that is, 
\begin{equation}\label{eq:010}
\frac{d}{d t} 
\int_{\Omega} ( \omega_{1} - \omega_{2} )^{2} \dx 
+ \norm{\nabla\left(\omega_1 - \omega_2\right)}_{\Ltwo(\Omega)}^{2}
\leq \Cone^{2} \norm{\omega_2}_{L^4(\Omega)}^{2}
\int_{\Omega} ( \omega_{1} - \omega_{2} )^{2} \dx . 
\end{equation}
Ignoring the second term of the left-hand side, we have 
\[
\int_{\Omega} ( \omega_{1} - \omega_{2} )^{2} \dx 
\leq \int_{\Omega} ( \omega_{01} - \omega_{02} )^{2} \dx 
\exp \left(
 \Cone^{2} \int_{0}^{t} \norm{\omega_{2} (s)}_{L^{4} ( \Omega )}^{2} \, d s 
\right) . 
\]
Taking the supremum over the time interval $(0,T)$, we have 
\[
\norm{\omega_{1} - \omega_{2}}_{L^{\infty} (0,T; \Ltwo ( \Omega ) )}^{2}
\leq \exp \left(
 \Cone^{2} \int_{0}^{T} \norm{\omega_{2} (t)}_{L^{4} ( \Omega )}^{2} \, d t
\right)
\norm{\omega_{01} - \omega_{02}}_{\Ltwo ( \Omega )}^{2} . 
\]
Here, by Lemma~\ref{lemma:2}, we have the estimate 
\begin{equation}\label{eq:012}
\int_{0}^{T} \norm{\omega_{2} (t)}_{L^{4} ( \Omega )}^{2} \, d t
\leq \Ctwo \max \{T, 1\} \norm{\omega_{02}}_{\Ltwo ( \Omega )}^{2} 
\leq \Ctwo \max \{T, 1\} M^{2},
\end{equation}
where the positive constant $\Ctwo$ depends only on $\Omega$. 
Then we have
\begin{equation}\label{eq:011}
\norm{\omega_{1} - \omega_{2}}_{L^{\infty} (0,T; \Ltwo ( \Omega ) )}^{2}
\leq \exp \left(
 \Cone^{2} \Ctwo \max \{T, 1\} M^{2}
\right)
\norm{\omega_{01} - \omega_{02}}_{\Ltwo ( \Omega )}^{2} . 
\end{equation}
On the other hand, integrating \eqref{eq:010} over the time interval $(0,T)$, 
we have 
{\[
\norm{\nabla (\omega_1 - \omega_2 )}_{\Ltwo (0,T;\Ltwo(\Omega))}^{2}
\leq \Cone^{2} \int_{0}^{T} \norm{\omega_2}_{L^4(\Omega)}^{2}
\norm{\omega_{1} - \omega_{2}}_{\Ltwo (\Omega)}^{2} \dt 
+ \norm{\omega_{01} - \omega_{02}}_{\Ltwo (\Omega)}^{2} . 
\]
Here, for the first term of the right-hand side, 
we have 
\begin{align*}
\int_{0}^{T} \norm{\omega_2}_{L^4(\Omega)}^{2}
\norm{\omega_{1} - \omega_{2}}_{\Ltwo (\Omega)}^{2} \dt 
& \leq 
\exp \left(
 \Cone^{2} \Ctwo \max \{T, 1\} M^{2}
\right)
\norm{\omega_{01} - \omega_{02}}_{\Ltwo ( \Omega )}^{2} 
\int_{0}^{T} \norm{\omega_2}_{L^4(\Omega)}^{2} \dt \\
& \leq 
\Ctwo \max \{T, 1\} M^{2} 
\exp \left(
 \Cone^{2} \Ctwo \max \{T, 1\} M^{2}
\right)
\norm{\omega_{01} - \omega_{02}}_{\Ltwo ( \Omega )}^{2} 
\end{align*}
by using \eqref{eq:011} and \eqref{eq:012}.
Hence, we have 
\begin{equation}\label{eq:013}
\norm{\nabla (\omega_1 - \omega_2 )}_{\Ltwo (0,T;\Ltwo(\Omega))}^{2}
\leq \Cthree \norm{\omega_{01} - \omega_{02}}_{\Ltwo ( \Omega )}^{2} , 
\end{equation}
where  
\begin{math}
\Cthree: = \Cone^{2} \Ctwo \max \{T, 1\} M^{2} 
\exp \left(
 \Cone^{2} \Ctwo \max \{T, 1\} M^{2}
\right)
+ 1
\end{math}. 
}

Next, we provide a stability estimate 
for the boundary vorticity difference $h_1 - h_2$. 
To begin with, we have
\begin{align*}
\int_{0}^{T} (h_1 - h_2)^2 \dt 
&= \frac{1}{|\Omega|} 
\int_{0}^{T} \int_{\Omega} (h_1 - h_2 )^2 \dx \dt \\
& \mbox{{\begin{math}\displaystyle
\mbox{} \leq \frac{2}{| \Omega |} \left( 
 \int_{0}^{T} \int_{\Omega} (\omega_1 - \omega_2 )^2\dx\dt 
 + \int_{0}^{T} \int_{\Omega}
  [ (\omega_1 - \omega_2 )- ( h_1 - h_2 )]^2
 \dx \dt 
\right) . 
\end{math}}}
\end{align*}
Here, as for the second term, we have
\[
\int_{\Omega}
 [ (\omega_1 - \omega_2 )- ( h_1 - h_2 )]^2
\dx 
\lesssim \int_{\Omega} 
 | \nabla [ (\omega_1 - \omega_2 )- ( h_1 - h_2 )] |^{2}
\dx 
\]
by the Poincar\'e inequality, due to the second condition of \eqref{eq:1.2}, it yields 
$( \omega_{1} - \omega_{2} ) - ( h_{1} - h_{2} ) = 0$ on $\partial \Omega$.
By these estimates and each $h_{i} (t)$ is independent of $x$, we have
\begin{equation}\label{eq:014}
\int_{0}^{T} (h_1 - h_2)^2 \dt 
\lesssim 
\int_{0}^{T} \int_{\Omega} (\omega_1 - \omega_2 )^2\dx\dt 
+ \int_{0}^{T} \int_{\Omega}
 | \nabla (\omega_1 - \omega_2 ) |^{2}
\dx \dt . 
\end{equation}
Combining 
{\eqref{eq:011}, \eqref{eq:013} 
and \eqref{eq:014},} 
we obtain the desired stability estimate 
\begin{equation}\label{eq:1.5}
\norm{\omega_1 - \omega_2}_{\Linf(0,T;\Ltwo(\Omega))} +\norm{\nabla (\omega_1 - \omega_2 )}_{\Ltwo (0,T;\Ltwo(\Omega))}+ \norm{h_1 - h_2}_{\Ltwo(0,T)} \leq \Cfour \norm{\omega_{01} - \omega_{02}}_{\Ltwo(\Omega)}
\end{equation}
for any $T>0$, where $\Cfour$ depends only on $\Omega, T$ and $M$. 

Note that the difference of the velocity fields satisfies $\divv ( u_{1} - u_{2} ) = 0$ in $\Omega$ and $( u_{1} - u_{2} ) \cdot n = 0$ on $\partial\Omega$. Then, applying the elliptic regularity estimate (2.11) in the proof of  Lemma 2.1, we have
\begin{equation*}
\|(u_1 - u_2)(t)\|_{H^1(\Omega)} \lesssim \|(\omega_1 - \omega_2)(t)\|_{L^2(\Omega)}.
\end{equation*}
Taking the supremum over $t$ on both sides, and using the first term on the left-hand side of \eqref{eq:1.5}, we have
\begin{equation} \label{eq:u_bound}
\|u_1 - u_2\|_{L^\infty(0,T; H^1(\Omega))} 
\leq \|\omega_1 - \omega_2\|_{L^\infty(0,T; L^2(\Omega))} 
\leq \Cfour \norm{\omega_{01} - \omega_{02}}_{\Ltwo(\Omega)} .
\end{equation}
Subtracting the momentum equations corresponding to $(u_1, p_1)$ and $(u_2, p_2)$ yields
\begin{equation} \label{eq:momentum_diff}
\partial_t (u_{1}-u_{2}) + \nabla(p_1 - p_2) = \Delta (u_{1}-u_{2}) - (u_1 \cdot \nabla u_1 - u_2 \cdot \nabla u_2).
\end{equation}
Taking the $L^2$ norm of equation \eqref{eq:momentum_diff} over $\Omega \times (0,T)$, utilizing the fact that
\begin{equation*}
\int_\Omega \partial_t (u_{1}-u_{2}) \cdot \nabla(p_1 - p_2) \, dx = 0
\end{equation*}
holds by the same argument as for \eqref{eq:1.13}, we obtain
\begin{equation*}
\int_0^T \int_\Omega 
 \bigl( |\partial_t (u_{1}-u_{2})|^2 + |\nabla(p_1 - p_2)|^2 \bigr) 
\dx \dt 
= \int_0^T \int_\Omega 
 |\Delta (u_{1}-u_{2}) - (u_1 \cdot \nabla u_1 - u_2 \cdot \nabla u_2)|^2 
\dx \dt.
\end{equation*}
This directly bounds the pressure gradient as follows
\begin{equation*}
\|\nabla(p_1 - p_2)\|_{L^2(0,T; L^2(\Omega))} 
\leq \|\Delta (u_{1}-u_{2})\|_{L^2(0,T; L^2(\Omega))} 
+ \|u_1 \cdot \nabla u_1 - u_2 \cdot \nabla u_2\|_{L^2(0,T; L^2(\Omega))}.
\end{equation*}
By the higher-order elliptic estimate \eqref{eq:1.11t}, we have $\|(u_{1}-u_{2})(t)\|_{H^2(\Omega)} \lesssim \|( \omega_{1} - \omega_{2} ) (t)\|_{H^1(\Omega)}$. Then, together with this estimate and \eqref{eq:1.5}, we have
\begin{align}
\|\Delta (u_{1}-u_{2})\|_{L^2(0,T; L^2(\Omega))} 
& \leq \|u_{1}-u_{2}\|_{L^2(0,T; H^2(\Omega))} 
\lesssim \|\omega_1 - \omega_2\|_{L^2(0,T; H^1(\Omega))} \notag \\
& \leq \Cfour \sqrt{T+1} 
\norm{\omega_{01} - \omega_{02}}_{\Ltwo(\Omega)} .
\label{eq:deltav}
\end{align}

Now, rewrite the nonlinear term as
\begin{equation*}
u_1 \cdot \nabla u_1 - u_2 \cdot \nabla u_2 
= u_1 \cdot \nabla ( u_{1} - u_{2} ) + ( u_{1} - u_{2} ) \cdot \nabla u_2.
\end{equation*}
Then, apply H\"{o}lder's inequality and the Sobolev embeddings, 
we estimate their $L^2(\Omega)$ norms as
\begin{align*}
\|u_1 \cdot \nabla ( u_{1} - u_{2} )\|_{L^2(\Omega)} 
&\leq \|u_1\|_{L^\infty(\Omega)} \|\nabla u_{1} - u_{2} \|_{L^2(\Omega)} 
\lesssim \|u_1\|_{H^2(\Omega)} \| u_{1} - u_{2} \|_{H^1(\Omega)}, \\
\|u_{1} - u_{2}  \cdot \nabla u_2\|_{L^2(\Omega)} 
&\leq \| u_{1} - u_{2} \|_{L^4(\Omega)} \|\nabla u_2\|_{L^4(\Omega)} 
\lesssim \| u_{1} - u_{2} \|_{H^1(\Omega)} \|u_2\|_{H^2(\Omega)}.
\end{align*}
Integrating the square of these norms over $(0,T)$, we have
\begin{align*}
\int_0^T \|u_1 \cdot \nabla ( u_{1} - u_{2} )\|_{L^2(\Omega)}^2 \, dt 
&\lesssim  \| u_{1} - u_{2} \|_{L^\infty(0,T; H^1(\Omega))}^2 
\|u_1\|_{L^2(0,T; H^2(\Omega))}^2, \\
\int_0^T \|( u_{1} - u_{2} ) \cdot \nabla u_2\|_{L^2(\Omega)}^2 \, dt 
&\lesssim \| u_{1} - u_{2} \|_{L^\infty(0,T; H^1(\Omega))}^2 
\|u_2\|_{L^2(0,T; H^2(\Omega))}^2.
\end{align*}
Since we have the bounds
$\|u_i\|_{L^2(0,T; H^2(\Omega))} \lesssim \max\{ \sqrt{T} , 1 \} M$ for $i=1,2$ by Lemma~\ref{lemma:1}, 
we have
\begin{equation} \label{eq:pressure_grad_bound}
\|u_1 \cdot \nabla u_1 - u_2 \cdot \nabla u_2\|_{L^2(0,T; L^2(\Omega))}
\leq \Cfour \max \{ \sqrt{T} , 1 \} M 
\norm{\omega_{01} - \omega_{02}}_{\Ltwo(\Omega)}
\end{equation}
by using \eqref{eq:u_bound}.
Combining \eqref{eq:deltav} and \eqref{eq:pressure_grad_bound}, 
we obtain
\begin{equation*}
\|\nabla(p_1 - p_2)\|_{L^2(0,T; L^2(\Omega))} 
\leq \Cfive \norm{\omega_{01} - \omega_{02}}_{\Ltwo(\Omega)} , 
\end{equation*}%
where $\Cfive$ depends only on $\Omega$, $T$ and $M$.
Since $\int_\Omega (p_1 - p_2) \, dx = 0$, Poincar\'{e}'s inequality yields
\begin{equation*}
\|p_1 - p_2\|_{L^2(0,T; L^2(\Omega))} 
\lesssim \|\nabla(p_1 - p_2)\|_{L^2(0,T; L^2(\Omega))}.
\end{equation*}
Hence, combining these two estimates, we have the desired estimate given as
\begin{equation} \label{eq:pressure_final_bound}
\|p_1 - p_2\|_{L^2(0,T; H^1(\Omega))} 
\lesssim \Cfive \norm{\omega_{01} - \omega_{02}}_{\Ltwo(\Omega)} . 
\end{equation}

Summing \eqref{eq:u_bound}, \eqref{eq:pressure_final_bound} and the boundary vorticity bound $\|h_1 - h_2\|_{L^2(0,T)}$ obtained directly from \eqref{eq:1.5}, we finally obtain the estimate \eqref{eq:1.5t} in Theorem \ref{thm:1.1}. This completes the proof.

\section*{Acknowledgments}
Regarding financial supports, the second author was supported by the National Natural Science Foundation of China (No. 12241103), and the fourth author was partially supported by JSPS KAKENHI (Grant No. JP25K07076). 

\printbibliography

\end{document}